\documentclass[11pt,reqno]{amsart}
\usepackage{mathrsfs}
\usepackage{amsfonts,amssymb,amsmath,amsthm}
 \usepackage{enumerate}
\newtheorem{theorem}{Theorem}[section]
\newtheorem{lemma}[theorem]{Lemma}

\theoremstyle{definition}
\newtheorem{definition}[theorem]{Definition}
\newtheorem{remark}{Remark}
\numberwithin{equation}{section}

\usepackage[left=2.0cm, right=2.0cm, top=2.0cm, bottom=2.0cm]{geometry}

\usepackage{cite}

\DeclareMathOperator{\R}{\mathbb{R}}

\author{Lixia Yuan}
\address{
Mathematics and Science College\\
Shanghai Normal University\\
 Shanghai, 200234, China}
\email{yuanlixia@shnu.edu.cn}

\author{Wei Zhao}
\address{
Department of Mathematics\\
East China University of Science and Technology\\
Shanghai, 200237,  China}
\email{szhao\underline{ }wei@yahoo.com}

\begin{document}

\title[On a Curvature Flow in a Band Domain with Unbounded Boundary Slopes]{On a Curvature Flow in a Band Domain with Unbounded Boundary Slopes}

\begin{abstract}
We consider an anisotropic curvature flow $V= A(\mathbf{n})H + B(\mathbf{n})$ in a band domain $\Omega :=[-1,1]\times \R$,
where $\mathbf{n}$, $V$ and $H$ denote the unit normal vector, normal velocity and
curvature, respectively, of a graphic curve $\Gamma_t$. We consider the case when $A>0>B$ and the curve $\Gamma_t$ contacts $\partial_\pm \Omega$  with
slopes equaling to $\pm 1$ times of its height (which are unbounded when the solution moves to infinity).
 First, we present the global well-posedness and then, under some symmetric assumptions on $A$ and $B$,
we show the uniform interior gradient estimates for the solution.
Based on these estimates, we prove that $\Gamma_t$ converges as $t\to \infty$ in $C^{2,1}_{\text{loc}} ((-1,1)\times \R)$ topology to a
cup-like traveling wave with {\it infinite} derivatives on the boundaries.
\end{abstract}
\maketitle

\section{Introduction} \label{sect1}
  Consider the following curvature flow
\begin{equation}
  V= A(\mathbf{n})H  + B(\mathbf{n})\qquad \text{on} \qquad\Gamma_t\subset \Omega, \label{mcf}
\end{equation}
in the band domain $\Omega:=\{(x,y)|-1\leq x\leq 1, y\in\R\}$ in $\R^2$. Here, $\Gamma_t$ is a family of simple curves
in $\Omega$ which contact the boundaries  with prescribed angles, the geometric quantities $\mathbf{n}$, $V$ and $H$ denote the upward unit normal vector, the normal velocity and
the curvature of $\Gamma_t$, respectively, and $A,-B$ are two smooth positive functions defined on $\mathbb{S}^1$. The equation \eqref{mcf} is an important model arising in physics such as phase transition problems, and in
singular limit problems of some partial differential equations (cf. \cite{GM,MW,RJ,CX,MD} etc). In particular,
if the curve $\Gamma_t$ is a smooth graph of a function $y=u(x,t)$  for each $t$, then
$$
\mathbf{n} = \frac{(-u_x, 1)}{\sqrt{1+u_x^2}},\qquad   V=\frac{u_t}{\sqrt{1+u^2_x}},\qquad H=\frac{u_{xx}}{\left(1+u^2_x\right)^{3/2}},
$$
and Problem (\ref{mcf}) can be expressed as
\begin{equation}\label{A}
\left\{
 \begin{array}{llll}
\displaystyle u_t= a(u_x) \frac{u_{xx}}{1+u^2_x} + b(u_x) \sqrt{1+u_x^2},&& -1<x<1,\quad t>0,\\
\\
u_x(-1,t)=g_-,\quad u_x(1,t)=g_+,&& t>0,
 \end{array}
 \right.
\end{equation}
where
$$
a(u_x) = A\left( \frac{(-u_x, 1)}{\sqrt{1+u_x^2}} \right),\quad b(u_x) = B \left( \frac{(-u_x, 1)}{\sqrt{1+u_x^2}} \right),
$$
and $g_-, g_+$ denote the boundary contact conditions induced by the prescribed angles.

In case  $a\equiv 1$ and $b\equiv 0$, a lot of effort has been devoted to the study of this problem.
In 1989, Huisken \cite{Hui} considered the boundary value problem for the equation in Problem \eqref{A}.
He proved that any global solution to the Dirichlet problem converges to a
linear function, while any global solution to the homogeneous Neumann problem converges
to a constant. In 1993, Altschuler and Wu \cite{AW1} studied the inhomogeneous Neumann
problem, that is, Problem \eqref{A} with $g_+, -g_-$ being positive constants. They
proved that any global solution converges to a {\it grim reaper} (which is also called a
traveling wave).  A year later this result was extended  to the case of two dimension by themselves (see Altschular and Wu \cite{AW2}).
If $a$ and $b$ are not constants, Cai and Lou \cite{CaiLou2} investigated Problem \eqref{A}
with $g_\pm $ being (almost) periodic functions of $u$, and proved that any solution
converges to a (almost) periodic traveling wave. Recently, Yuan and Lou \cite{YuanLou} have considered the problem
when $g_\pm = g_\pm (u)$ are asymptotic periodic functions as $u\to \pm \infty$.
They constructed some entire solutions connecting two periodic traveling waves.
There are many other interesting works related to the mean curvature flow  in domains with boundaries; for example,  see \cite{LMN, MNL}, etc.  for problems in band domain with undulating boundaries; see \cite{CGK, GGH, GH, Koh}, etc. for self-similar solutions in sectors on the plane; and see \cite{ChenGuo, GMSW}, etc.  for problems on the half space.

In all the works mentioned above, the boundary slopes   are bounded, no matter when the equations are linear or  nonlinear.
In this paper, we consider the case of unbounded boundary slopes. More precisely,
\begin{equation}\label{B}
\left\{
 \begin{array}{ll}
\displaystyle u_t= a(u_x) \frac{u_{xx}}{1+u^2_x} + b(u_x) \sqrt{1+u_x^2}, &  -1<x<1,\quad t>0,\\
\\
u_x(\pm1,t)=\pm u(\pm1,t), & t>0,\\
\\
u(x,0)=u_0(x), & -1\leq x\leq 1.
 \end{array}
 \right.
\end{equation}

In 2012, Chou and Wang \cite{CW} studied the equation in \eqref{B} under the assumptions that $a\equiv 1,\ b\equiv 0$ and Robin boundary conditions:
$$
u_x(\pm 1,t) = \alpha_\pm u(\pm 1,t) +\beta_\pm,\quad t>0.
$$
They divided the parameters $\alpha_\pm$ and $\beta_\pm$ into several cases, and
studied the asymptotic behavior in each cases. But for the cases $\min u\to \infty$
or $\max u\to -\infty$ (this happens, for example, when $\alpha_- <0 <\alpha_+$),
they did not obtain the convergence of the solution and left it as an open problem. Recently, this problem has been solved by
Lou, Wang and Yuan \cite{LWY}.
More precisely,
in case $a\equiv 1$ and $b\equiv 0$, they proved that $u$ converges to
a {grim reaper} with span $(-1,1)$.

The present paper is devoted to Problem \eqref{B} with non-constant
$a$ and $b$. We show that in this case the {grim reaper}  no longer exist  while some
{\it cup-like} traveling waves can be constructed by shooting method. In particular,
the profile of such a  wave has a finite height, which is different from that of a grim reaper. See Sections  \ref{sec4}-\ref{sec5} below for details.

In order to state our main result, we first
give a result concerned with the following ODE: given $h\in\R\cup\{+\infty\}$, find a solution pair $(c,\varphi)$ to
\begin{equation}\label{TW-p}
\left\{
 \begin{array}{l}
 \displaystyle c =a(\varphi') \frac{\varphi'' }{1+ (\varphi')^2} + b(\varphi') \sqrt{1+ (\varphi')^2}, \qquad x\in (-1,1),\\
 \\
 \varphi'(-1)= -h,\quad \varphi'(1)= h.
 \end{array}
 \right.
\end{equation}

For convenience, set
\begin{equation*}\label{def-a0b0}
a^0 := \max a \geq a_0 := \min a,\qquad   b_0 := \min b \leq b^0 := \max b.
\end{equation*}
Then we have the following result, which plays an important role in studying Problem \eqref{B}.
\begin{theorem}\label{thm:exist-GR}
Assume $a(p)>0>b(p)$ for any $p\in \R$.
\begin{itemize}
\item[\rm (i)] If $a_0 > -b_0$, then Problem \eqref{TW-p} with $h=+\infty$ has a solution pair
$(c,\varphi) = (\bar{c}, \Phi(x))$. In particular, $\bar{c}$ is unique with $0<\bar{c}<\frac{\pi a^0}{2}$ while $\Phi(x)$ is unique up to a shift.
\item[\rm (ii)] If $h>0$ satisfies $a_0 h > -b_0 \sqrt{1+h^2}$, then Problem \eqref{TW-p} has a solution pair
$(c,\varphi) = (c(h), \Phi(x;h))$.  Furthermore, $c(h)$ is unique and strictly increasing in $h$ with $0<c(h)< a^0 \arctan h$ while $\Phi(x;h)$ is unique up to a shift.
\end{itemize}
\end{theorem}

Note that a
solution pair $(c,\varphi)$ to Problem \eqref{TW-p} gives a traveling wave to  Problem \eqref{A} in the form of $u=\varphi(x) + ct$. In this paper,
a traveling wave $\varphi(x;c)+ct$ derived from  Theorem \ref{thm:exist-GR} is called a {\it cup-like traveling wave} since
the graph of $\varphi$ is similar to a cup with finite height.
With the help of this kind of solutions, we solve Problem \eqref{B}. More precisely,

\begin{theorem}\label{thm:main}
Suppose that $a,b$ are two even functions with
$a(p)>0>b(p)$ for $p\in \R$ and $a_0 > -b_0$.
If $u_0$ is a $C^1$-function satisfying compatibility conditions, then Problem \eqref{B} has a time-global solution $u(x,t)$. Moreover, if $u_0$ is large enough, then $u(x,t)$ tends to infinity as $t\rightarrow +\infty$, and its profile satisfies
\begin{equation}\label{1.5thereom2secnd}
 u(x,t+s) -u(0,s) \to  \Phi(x) + \bar{c} t, \quad \mbox{as}\ s\to +\infty,
\end{equation}
in the topology of $C^{2,1}_{\text{loc}}\left( (-1,1)\times \R\right),$ where $(\bar{c},\Phi(x))$ is the solution pair in Theorem \ref{thm:exist-GR}/(i).
\end{theorem}

This paper is organized  as follows. Section \ref{sec2} is devoted to traveling waves, in which Theorem \ref{thm:exist-GR} is proved. In Section \ref{sec3}, we present a
priori estimates  and show the time-global existence for the solution to Problem  \eqref{B}.
We devote Section \ref{sec4} to the symmetric solutions to Problem \eqref{B}. First we obtain
precise estimates for $u_x$ by the zero number argument, and then we show the convergence of $u$ to
$\Phi(x)+\bar{c}t$.
The general solutions are considered in Section  \ref{sec5}, where Theorem \ref{thm:main} is proved.

%%%%%%%%%%%%%%%%%%%%%%%%%%%%%%%%%%%%
\section{Traveling waves}\label{sec2}
In this section, we study Problem \eqref{TW-p} and prove Theorem \ref{thm:exist-GR}. Note that every solution pair $(c,\varphi)$ to Problem \eqref{TW-p} induces a  traveling waves $u=\varphi+ct$ to Problem \eqref{A}.
For convenience, the notation  \eqref{TW-p}$_1$ is used to denote the equation in \eqref{TW-p}.
Now we recall the following result.

\begin{lemma}[Lou \cite{Lou1}]\label{lem:homo}
Assume $a$ and $-b$ are positive constants. Then for any $c>0$, the equation \eqref{TW-p}$_1$
has a solution
$\varphi = \Phi(x;a,b,c)$ with a cup-like graph, that is, $\Phi(x;a,b,c)$ is defined in $[-X, X]$ for some $X = X(a,b,c)>0$ staisfying
$$
\Phi(0)=0,\quad \Phi(-x)=\Phi(x) \mbox{ and } \Phi''(x)>0 \mbox{ in } [-X,X],\quad \Phi'(\pm X) = \pm \infty,\quad \Phi(\pm X) <\infty.
$$
\end{lemma}

Lemma \ref{lem:homo} implies that the equation \eqref{TW-p}$_1$ with $a\equiv a_0, b\equiv b_0$ (resp.,
$a\equiv a^0, b\equiv b^0$) has a solution with cup-like graph, which is denote by
 $\Phi_0(x;c)= \Phi(x;a_0,b_0,c)$ in $[-X(a_0,b_0,c), X(a_0,b_0,c)]$
(resp., $\Phi^0(x;c)= \Phi(x;a^0,b^0,c)$ in $[-X(a^0,b^0,c), X(a^0,b^0,c)]$).
We are going to investigate  the equation \eqref{TW-p}$_1$ with non-constant $a, b$.
%For this purpose, we first consider \eqref{TW-p}$_1$  under the following initial conditions:

\begin{lemma}\label{lem:prop+varphi}Let $a,b$ be two functions with $a>0>b$.
Given any $c>0$, let $\varphi(x;c)$ be the unique solution to the equation \eqref{TW-p}$_1$ satisfying the initial conditions
\begin{equation}\label{ini-condition}
\varphi(0)=0,\quad \varphi'(0)=0.
\end{equation}
Let $(X^-(c),X^+(c))$ denote the maximal existence interval of $\varphi(x;c)$.
Then we have
\begin{itemize}
\item[\rm (1).] $\varphi''(x;c)>0$ for any $x\in [0, X^+(c))$;
\item[\rm (2).] $X(a_0, b_0,c) \leq  X^+(c) \leq  X(a^0, b^0,c)$. In particular, $\varphi'(x;c)\leq \Phi'_0(x;c)$ for any  $x\in (0, X(a_0, b_0,c))$, and
$\varphi'(x;c) \geq (\Phi^0)'(x;c)$ for any $x\in (0, X^+(c))$;
\item[\rm (3).] $\lim_{x\to X^+(c)} \varphi(x;c)\leq \frac{a^0}{-b^0}$ and $\lim_{x\to X^+(c)} \varphi'(x;c) = +\infty$;
\item[\rm (4).] $\varphi'(x;c)$ is strictly increasing in $c$ while $X^+(c)$ is strictly decreasing in $c$ with $\lim_{c\rightarrow+\infty}X^+(c)= 0$;
\item[\rm (5).] under the additional condition
\begin{equation}\label{prop-cond-right}
\int^{\infty}_0\frac{a(r)\; dr}{-b(r)(1+r^2)^{\frac{3}{2}}} >1,
\end{equation}
there holds $X^+(c) > 1$ for $0\leq c\ll 1$.
\end{itemize}
\end{lemma}
%%%%%%%%%%%%%%%%%%%%%%%%%%%%%%%%   Section 3  %%%%%%%%%%%%%%%%%%%%%%%%%%%%%%%%%%%%%%%
\begin{proof}
(1). The equation (\ref{TW-p})$_1$ can be re-written as
\begin{equation}\label{varphi''}
\varphi'' =\frac{ 1+ (\varphi')^2 }{a(\varphi') } \left[c- b(\varphi') \sqrt{1+ (\varphi')^2}\right].
\end{equation}
It follows from $a, c>0$ and $b<0$ that $\varphi''(x;c)>0$ in $[0,X^+(c))$.

\medskip
(2). For any small $\varepsilon>0$, consider an auxiliary problem \eqref{varphi''} with the  initial conditions \eqref{ini-condition}. Replacing  $c$  by $c-\varepsilon>0$,  we have
\begin{equation}
\varphi'' (0;c-\varepsilon)=\frac{ 1}{a(0) } \left[c-\varepsilon - b(0)\right]<\frac{ 1}{a_0 } \left[c- b_0\right]=\Phi'' _0 (0;c).
\end{equation}
Therefore, $\varphi'(x;c-\varepsilon)<\Phi'_0(x;c)$ for $0<x \ll 1$.

 We claim that $\varphi'(x;c-\varepsilon)<\Phi'_0 (x;c)$ holds in their common existence interval in $(0,\infty)$. If not,
let $x_1>0$ be the smallest number with $\varphi'(x_1;c-\varepsilon)=\Phi'_0(x_1;c)$.
 Then we also have $\varphi''(x_1;c-\varepsilon)\geq\Phi''_0(x_1;c)$.
 However,  the equation (\ref{varphi''}) (with $c$ being replaced by $c-\varepsilon$) yields
 \begin{eqnarray*}
 \varphi''(x_1;c-\varepsilon) &= & \displaystyle \frac{ 1+ (\varphi'(x_1;c-\varepsilon))^2 }{a(\varphi'(x_1;c-\varepsilon)) } \left[c -\varepsilon - b(\varphi'(x_1;c-\varepsilon)) \sqrt{1+ (\varphi'(x_1;c-\varepsilon))^2}\right]\\
  &< & \displaystyle \frac{ 1+ (\Phi'_0(x_1;c))^2 }{a_0 } \left[c- b_0\sqrt{1+ (\Phi'_0(x_1;c))^2}\right]
=\Phi''_0(x_1;c),
 \end{eqnarray*}
which is a contradiction. Hence, the claim is true and therefore,
$\Phi'_0 (x;c) \geq \varphi'(x;c)$ in their common existence interval.
This implies that the maximal existence interval of $\Phi_0(x;c)$ is not wider than that of $\varphi(x;c)$ in $(0,\infty)$,
that is, $X(a_0,b_0,c) \leq X^+(c)$. A similar argument yields that $\varphi'(x;c) \geq (\Phi^0)'(x;c)$ for $x\in [0,X^+(c)) \subset [0, X(a^0,b^0, c))$.

%\fbox{Note to you: two functions $f(x)= 2\tan x$ and $g(x)= \tan x$ satisfy $f'(x)>g'(x)$, but with same span $[0,\frac{\pi}{2})$

\medskip
(3). It follows from Lemma \ref{lem:homo} that $(\Phi^0)'(x;c)\to +\infty$ as $x\to X(a^0,b^0,c)$. Thus the result in
 the above step furnishes $\varphi'(x;c)\to +\infty$ as $x\to X^+(c)$.

Now we claim $\lim_{x\rightarrow X^+(c)}\varphi(x;c) \leq R := \frac{a^0 }{-b^0}$. Suppose by contradiction that for some small $\delta>0$,
\begin{equation}\label{var}
\lim_{x\to X^+(c)}\varphi\left(x;c\right) > R +2\delta.
\end{equation}
 Let $\rho (x)$ denote
the lower half of the circle of radius $R$  centered at $(-R,R+\delta)$, i.e.,
$$
\rho (x) := R +\delta- \sqrt{ R^2  - ( x+R)^2 },\qquad -2R  \leq x \leq 0.
$$
Clearly, this half circle (denoted by $\mathcal{C}_0$) has no contact points with the graph of  $\varphi(x;c)$ (denoted by $\mathcal{C}$).
We now move $\mathcal{C}_0$ rightward little by little till it just touches $\mathcal{C}$. More precisely, set
$$
d_1 := \max \{d>0 \mid \rho (x-d) > \varphi(x;c) \mbox{ in their common domain}\}.
$$
Then the graph of $\rho(x-d_1)$ (denoted by $\mathcal{C}_1$) lies above $\mathcal{C}$, which is  tangent to $\mathcal{C}$ at some point  $(x_2, \varphi(x_2;c))$.  Therefore, at this point, the curvature of $\mathcal{C}_1$ is not smaller than
that of $\mathcal{C}$, that is,
$$
\frac{-b^0}{a^0} =\frac{1}{R} \geq \frac{\varphi''(x_2;c)}{[1+(\varphi'(x_2;c))^2]^{3/2}} = \frac{1}{a(\varphi'(x_2;c))} \left[
\frac{c}{\sqrt{1+(\varphi'(x_2;c))^2}} - b(\varphi'(x_2;c))\right] >\frac{-b^0}{a^0},
$$
which is a contradiction. Hence, the claim is true.

\medskip
(4). Set $\psi(x):= \varphi'(x;c)$ and rewrite \eqref{TW-p}$_1$ as
\begin{equation}\label{psi}
dx=\frac{a(\psi)\; d\psi}{(1+\psi^2)\left(c-b(\psi)\sqrt{1+\psi^2}\right)}.
\end{equation}
Regard \eqref{psi} as an ODE of $x(\psi;c)$ with the initial condition $x(0)=0$.
By  the comparison principle, it is not hard to check that $x(\psi;c)$ is strictly decreasing in $c$. Therefore,
its inverse function $\psi(x) = \varphi'(x;c)$ is strictly increasing in $c$.

Next, by integrating \eqref{psi} over $x\in (0,X^+(c))$ (equivalently, over $\psi \in (0,\infty)$) we have
\begin{equation}\label{X+}
X^+(c) =\int^{\infty}_0\frac{a(r)\; dr}{(1+r^2)\left(c-b(r)\sqrt{1+r^2}\right)} <\int^{\infty}_0\frac{a^0\; dr}{c(1+r^2)}=\frac{ a^0\pi}{2c}.
\end{equation}
Thus, one gets
$$%\begin{equation}
\frac{dX^+(c) }{dc}=\int^{\infty}_0\frac{-a(r)\; dr}{(1+r^2)\left(c-b(r)\sqrt{1+r^2}\right)^2}<0,
$$%\end{equation}
which together with \eqref{X+} furnishes
$X^+(c)\to 0$ as $c\to +\infty$.

\medskip
(5). According to (\ref{X+}) and  \eqref{prop-cond-right}, we have
$$
X^+(0)= \int^{\infty}_0\frac{a(r)\; dr}{-b(r)(1+r^2)^{\frac{3}{2}}} >1.
$$%\end{align}
Then the continuity implies $X^+(c)>1$ for $0<c\ll 1$, which concludes the proof.
\end{proof}

\begin{remark}
Using Lemma \ref{lem:prop+varphi}/(3),  we can even supplementally define
 $\varphi(x;c)$ at the point $x=X^+(c)$ so that $\varphi(x;c)$ is continuous in $[0,X^+(c)]$.
\end{remark}

The same argument as in Lemma \ref{lem:prop+varphi}/(4) yields
\begin{equation}\label{X-}
X^-(c) := - \int^0_{-\infty}\frac{a(r)\; dr}{(1+r^2)\left(c-b(r)\sqrt{1+r^2}\right)} > - \frac{a^0 \pi}{2 c}.
\end{equation}
 By a suitable modification to the proof of Lemma \ref{lem:prop+varphi}, one can show the following result.

\begin{lemma}\label{lem:prop-varphi}Let $a,b$ be two functions with $a > 0 > b$.
For any $c>0$, let $\varphi(x;c)$ be the unique solution to the equation \eqref{TW-p}$_1$ with the initial conditions \eqref{ini-condition} and let $(X^-(c),X^+(c))$ be the
maximal existence interval of $\varphi(x;c)$. Then we have
\begin{itemize}
\item[\rm (1).] $\varphi''(x;c)>0$ for any $x\in( X^-(c), 0]$;
\item[\rm (2).] $X(a_0, b_0,c) \leq  - X^- (c) \leq  X(a^0, b^0,c)$. In particular,  $\varphi'(x;c)\geq \Phi'_0(x;c)$ for $x\in(-X(a_0, b_0,c),0)$ while
$\varphi'(x;c) \leq (\Phi^0)'(x;c)$ for $x\in(X^-(c),0)$;
\item[\rm (3).] $\lim_{x\to X^-(c)} \varphi(x;c)\leq \frac{a^0}{-b^0}$ and $\lim_{x\to X^-(c)} \varphi'(x;c) = -\infty$;
\item[\rm (4).] $\varphi'(x;c)$ is strictly decreasing in $c$ while $X^-(c)$ is strictly increasing in $c$ with $\lim_{c\rightarrow+\infty}X^-(c)=0$;
\item[\rm (5).] under the additional condition
\begin{equation}\label{prop-cond-left}
\int^0_{-\infty}\frac{a(r)\; dr}{-b(r)(1+r^2)^{\frac{3}{2}}} >1,
\end{equation}
there holds $X^-(c) <- 1$ for $0\leq c\ll 1$.
\end{itemize}
\end{lemma}

\begin{remark}\label{rem:prop-cond}The reason why we need \eqref{prop-cond-right} and \eqref{prop-cond-left}  is that
they can help
to construct the solutions to Problem \eqref{TW-p} for large $h$. On the other hand,
$a_0 > -b_0$ (in particular, $b\equiv 0$) is a sufficient condition for both \eqref{prop-cond-right} and \eqref{prop-cond-left}. In fact,
$$
\int^{\infty}_0 \frac{a(r)\; dr}{-b(r)(1+r^2)^{\frac{3}{2}}} \geq \frac{a_0}{-b_0} \int^{\infty}_0 \frac{dr}{(1+r^2)^{\frac{3}{2}}} = \frac{a_0}{-b_0}>1.
$$
\end{remark}

Based on the above two lemmas we now prove  Theorem \ref{thm:exist-GR}.

\begin{proof}[Proof of Theorem \ref{thm:exist-GR}]
 (i). According to Remark \ref{rem:prop-cond},  both \eqref{prop-cond-right} and \eqref{prop-cond-left} hold. Thus, it follows from Lemmas \ref{lem:prop+varphi}-\ref{lem:prop-varphi} that $d(c) := X^+(c)-X^-(c)$ is strictly decreasing
in $c$ and stasifies $\lim_{c\rightarrow+\infty}d(c)= 0$ and $d(c) > 2$ for $0<c \ll 1$.
Hence, there is a unique $c= \bar{c}$ with $d(\bar{c})=2$. Now \eqref{X+} together with
\eqref{X-} furnishes
$$
2 = d(\bar{c}) = X^+(\bar{c}) - X^-(\bar{c}) < \frac{a^0 \pi}{\bar{c}},
$$
which implies $0< \bar{c} < \frac{a^0 \pi}{2}$.

Denote by $\varphi(x;\bar{c})$ the unique solution to the equation \eqref{TW-p}$_1$ with the initial conditions \eqref{ini-condition} (where $c=\bar{c}$).
Set $\bar{x} := 1-  X^+(\bar{c})$ and $\Phi (x) := \varphi(x -\bar{x}; \bar{c})$. Then $\Phi(x)$ is a solution to
Problem  \eqref{TW-p} with $h=+\infty$. In particular, $\Phi$
is unique up to a vertical shift because  for any $f\in \R$,
$\Phi(x)+f$ is a solution to Problem  \eqref{TW-p}.

\smallskip

(ii). Integrating \eqref{psi} over $\psi\in (0,h)$ and $\psi\in (-h,0)$, respectively,  we have
\begin{equation*}\label{def-Xhc}
X^+_h (c) := \int^h_{0}\frac{a(r)\; dr}{(1+r^2)\left(c-b(r)\sqrt{1+r^2}\right)} , \quad
X_{h}^-(c) := - \int^0_{-h}\frac{a(r)\; dr}{(1+r^2)\left(c-b(r)\sqrt{1+r^2}\right)}.
\end{equation*}
From \eqref{psi} one can derive $\varphi'(X^\pm_h(c);c)= \pm h$.
Moreover, it is not hard to check that  $d_h (c) := X_{h}^+(c) - X_{h}^-(c)$ is strictly decreasing in $c$ and satisfies $\lim_{c\rightarrow+\infty}d_h (c)= 0$.
Note that the assumption implies
$$
d_h (0) = X^+_h (0) - X^-_h (0) > \frac{2a_0 h}{-b_0 \sqrt{1+h^2}}> 2.
$$
Thus,
there exists a unique $c= c(h)$ such that
\begin{equation}\label{find-ch}
d_h (c(h))= \int^h_{-h}\frac{a(r)\; dr}{(1+r^2)\left(c(h)-b(r)\sqrt{1+r^2}\right)} = 2,
\end{equation}
which implies that  $c(h)$ is strictly increasing in $h$. Moreover, we have
\[
d_h (c)\leq 2\max\{X^+_h (c), -X^-_h (c)\} <  2\frac{a^0 \arctan h}{c},
\]
which together with (\ref{find-ch}) yields $0 < c(h) < a^0 \arctan h$.

Finally, set $\tilde{x}:= 1- X^+_h (c(h)) $ and $\Phi(x;h) := \varphi(x - \tilde{x}; c(h))$. Then the domain
of $\Phi(x;h)$ contains $[-1,1]$, and
$$
\Phi'(\pm 1; h) = \varphi' (X^\pm_h (c(h)); c(h)) = \pm h.
$$
Hence,  $(c,\varphi) = (c(h), \Phi(x;h))$ is a solution pair to Problem  \eqref{TW-p},
which completes the proof.
\end{proof}

\section{Global well-posedness}\label{sec3}
In the sequel, we always assume that $a_0 > -b_0$.
\begin{definition}
A smooth function $\underline{u}(x,t)$ called a {\it lower solution to Problem (\ref{B})} if it satisfies
$$
\left\{
 \begin{array}{llll}
\displaystyle \underline{u}_t\leq a(\underline{u}_x) \frac{\underline{u}_{xx}}{1+\underline{u}^2_{x}} + b(\underline{u}_x)
\sqrt{1+ \underline{u}_x^2},\qquad -1<x<1,\ t>0,\\
\\
\underline{u}_x(1,t)\leq \underline{u}(1,t),\quad\underline{u}_x(-1,t)\geq -\underline{u}(-1,t),\qquad t>0.
 \end{array}
 \right.
$$
 A smooth function $\overline{u}(x,t)$ is called an {\it upper solution
of Problem \eqref{B}} if it satisfies the reverse inequalities.
\end{definition}

In order the study the asymptotic behavior of the solution $u$ to Problem \eqref{B},
we now present some sufficient conditions for  $\lim_{t\rightarrow+\infty}u(x,t)=+\infty$. Let $\varphi(x;0)$ be the solution to the equation
\eqref{TW-p}$_1$ with the initial conditions \eqref{ini-condition} (where $c=0$).  Lemmas \ref{lem:prop+varphi} and \ref{lem:prop-varphi} yield $\min\{X^+(0), -X^-(0)\} >1$ (since
$a_0 > -b_0$). Hence $\varphi(x;0)$ is well-defined over $[-1,1]$.
Let $M\in \R$ be the smallest number such that
\begin{equation}\label{def-M}
\varphi'(1;0)\leq \varphi(1;0)+M,\qquad \varphi'(-1;0)\geq -[\varphi(-1;0)+M].
\end{equation}

Recall that $u(x,0)=u_0(x)$ (see Problem \eqref{B}). Thus
the main result in the section reads as follows.
\begin{theorem}\label{thm:global}
Assume $a_0 >-b_0$. If $u_0(x)\in C^1([-1,1])$ satisfies
\begin{equation}\label{cond-u0}
u'_0(\pm 1) = \pm u_0(\pm 1),\quad u_0(x)> \varphi(x;0)+M \mbox{ in } [-1,1],
\end{equation}
then there is  a unique time-global classical solution $u(x,t)$ to Problem \eqref{B}. Moreover, $u(x,t)\to +\infty$
as $t\to +\infty$.
\end{theorem}

To prove this theorem we need some a priori estimates. The first one is the $L^{\infty}$-estimate.
\begin{lemma}\label{bound-est}
Let $u(x,t)$ be the classical solution in $[0,T]$ to Problem \eqref{B} with $u_0(x)$ satisfying \eqref{cond-u0}.
Then there exist positive constants $c_0,C_1,C_2>0$ (independent of $T$) such that
 \begin{equation}\label{3.1}
 c_0 t -C_1 \leq u(x,t)\leq \bar{c} t + C_2,\qquad x\in[-1,1],\ t\in[0,T],
 \end{equation}
 where $\bar{c}$ is the constant in Theorem \ref{thm:exist-GR}/(i).
\end{lemma}

\begin{proof}
The inequality in \eqref{cond-u0} implies that $u_0(x)\geq \varphi(x;0)+M+2\varepsilon$
for some small $\varepsilon>0$.
Since $\varphi(x;c)$ depends continuously on $c$, we can choose a small $c_0>0$ such that
$$
u_0(x) \geq \varphi(x;0)+M+2\varepsilon  > \varphi(x;c_0) +M+ \varepsilon \mbox{ in } [-1,1].
$$
By \eqref{def-M}, we can also assume that
$$
\varphi'(1;c_0) < \varphi(1;c_0)+M+ \varepsilon,\quad  \varphi'(-1;c_0) > -[\varphi(-1;c_0)+M+ \varepsilon].
$$
Hence, $\varphi(x;c_0)+M+\varepsilon+c_0 t $ is a lower solution to Problem \eqref{B}. Then the comparison principle yields
$$
u(x,t)\geq \varphi(x;c_0)+M+ \varepsilon +c_0 t,\quad x\in [-1,1],\ t\in [0,T].
$$
This proves the left-hand side of \eqref{3.1}. To prove the right-hand side, one only needs to
verify that
$$
\overline{u}(x,t) := \Phi(x) + \bar{c}t +\|u_0\|_{L^\infty }
$$
is an  upper solution to Problem \eqref{B}.
\end{proof}

Secondly, we needs  the following gradient estimate.
\begin{lemma}\label{gradientestm}
Let $u(x,t)$ be the solution to Problem  \eqref{B} in $[0,T]$. Then there exist $C_3(T)$ such that
$$
\left|u_x(x,t)\right|\leq C_3(T), \qquad x\in[-1,1],\ t\in[0,T].$$
\end{lemma}

\begin{proof}
From the above lemma, we see that
$$
\left|u_x(\pm 1,t)\right|=\left|u(\pm1,t)\right| \leq C_1 + C_2 + \bar{c} T , \qquad t\in[0,T].
$$
Then the maximum principle for $u_x$ yields
$$
\left|u_x(x,t)\right|\leq C_3 (T):=\max\{\|u'_0\|_{L^{\infty}}, C_1 + C_2 + \bar{c} T\}, \qquad x\in[-1,1],\ t\in[0,T],
$$
which concludes the proof.
\end{proof}

\begin{proof}[Proof of Theorem \ref{thm:global}] The standard parabolic theory together with
  the above a priori estimates (i.e., Lemmas \ref{bound-est}-\ref{gradientestm}) furnishes the time-global existence of the classical solution
$u(x,t)$ to Problem \eqref{B}.
The uniqueness of solutions can be proved in the standard way by the maximum principle.
By the left-hand side of \eqref{3.1},  we have
$u(x,t)\to +\infty$ as $t\to +\infty$.
\end{proof}

%%%%%%%%%%%%%%%%%%%%%%%%%%%%%%%%%%%%

%%%%%%%%%%%%%%%%%%%%%%%%%%%%%%%%   Section 4  %%%%%%%%%%%%%%%%%%%%%%%%%%%%%%%%%%%%%%%

\section{Symmetric solutions to Problem \eqref{B}} \label{sec4}
This section is devoted to the convergence of symmetric solutions to Problem \eqref{B}. The usual method of studying the convergence  is to find  a uniform global gradient estimate. However, it does not work here due to the boundary conditions in \eqref{B}.
Hence, we turn to investigate the {\it uniform} (in $t$) interior gradient estimates and the convergence in the $L^\infty_{\text{loc}}((-1,1))$-topology by means of the maximum principle and the so-called zero number argument
(i.e., zero number diminishing properties, cf. \cite{Ang, Lou2} for instance).

Throughout this section, both $a$ and $b$ are even functions, i.e.,
\begin{equation}\label{symm cond}
A\left( \frac{(-p,1)}{\sqrt{1+p^2}}\right) = A\left( \frac{(p,1)}{\sqrt{1+p^2}}\right),\quad
B\left( \frac{(-p,1)}{\sqrt{1+p^2}}\right) = B\left( \frac{(p,1)}{\sqrt{1+p^2}}\right),\quad p\in \R.
\end{equation}
In this case,   the traveling waves obtained by Theorem \ref{thm:exist-GR} are even functions as well:
\begin{equation}\label{phi}
\Phi(x)=\Phi(-x),\quad \Phi(x;h)=\Phi(-x;h),\quad x\in [-1,1].
\end{equation}
From these traveling waves, we will
derive the interior gradient estimates and asymptotic behaviors of symmetric solutions of Problem \eqref{B}.

%%%%%%%%%%%%%%%%%%%%%%%%%%%%%%%%%%%%%%%%%%

\subsection{Choice of initial data}
Choose an arbitrary constant pair $(p,M_1)$ with
\[
p\in [0,1],\ p\Phi'(\pm p)=\pm \left[\Phi(\pm p)+M_1\right],\ M_1>\varphi(x;0)+M,
\]
where $\Phi(x)$ and $M$ are defined   in \eqref{phi} and \eqref{def-M}, respectively. For such a pair $(p,M_1)$, we define
\begin{equation}\label{varphi-extension}
\rho(x) :=\rho_{p,M_1}(x) := \Phi( p x )+ M_1, \quad  x\in [-1, 1].
\end{equation}
Thus, $\rho(x)$
is a horizontal extension of $\Phi(x)$ (with vertical shift $M_1$) with
\begin{equation}\label{special-ini-data}
\left\{
\begin{array}{l}
   \rho\in C^2([-1,1]),\quad \rho'(x)<\Phi'(x) \mbox{ for } x\in (0,1], \quad \rho'(x)>\Phi'(x) \mbox{ for } x\in [-1,0),\\
   \\
   \rho'(\pm 1) = \pm \rho (\pm 1),\quad  \rho''(x) >0 ,\quad \rho(x) > \varphi(x;0)+M \ \mbox{ for } x\in [-1,1].
\end{array}
\right.
\end{equation}

In the sequel, we consider Problem \eqref{B} with the special initial data $u_0 =\rho$. In this case
Problem \eqref{B} has a unique time-global solution $u(x,t;\rho)$ which approaches infinity as $t\to +\infty$ (see Theorem \ref{thm:global}).

\subsection{Convexity of the solution}
\begin{lemma}\label{lem:convexity}
$u_{xx}(x,t;\rho)>0$ for all $x\in (-1,1),\ t>0$.
\end{lemma}

\begin{proof}For simplicity, we use \eqref{B}$_1$ to denote the equation in \eqref{B}. The proof is divided into two steps.

\smallskip

\textbf{Step 1.} It is easy to see $u_x$ satisfies a linear parabolic equation by
differentiating \eqref{B}$_1$ with respect to $x$.
Using the maximum principle to this problem, we conclude that the positive maximum of $u_x(\cdot,t;\rho)$ is
attained at $x=1$ while its negative minimum is attained at $x=-1$. We now claim $u_{xx}(1,t;\rho)\geq 0$. If not, one would have $u_x(x_1,t;\rho) > u_x(1,t;\rho)$
for some $x_1$ with $0\ll x_1 <1$, which contradicts  the positive maximum attained at boundary. Similarly, we have $u_{xx}(-1,t;\rho)\geq 0$.

\smallskip

\textbf{Step 2.} Differentiating \eqref{B}$_1$ twice we obtain
$$
(u_{xx})_t = a_1 (u_{xx})_{xx} + b_1 (u_{xx})_x + c_1 (u_{xx}),\quad -1 <x< 1,
$$
for some bounded $a_1, b_1, c_1$. Using the non-negativity of $u_{xx}$ on the boundaries $x=\pm 1$ we
conclude that $u_{xx}(x,t;\rho)\geq 0$ in $[-1,1]\times [0,\infty)$. Now the lemma  follows from the strong maximum principle.
\end{proof}

\subsection{Finer upper gradient estimate}

%In order to study the convergence of the profile of $u$, we need
%finer gradient estimates, especially, near the boundaries.

\begin{lemma}\label{lem:upper-interior-est}
The gradient of $u$ is bounded by $\Phi'(x)$ in the following sense:
\begin{equation}\label{upper-est}
\Phi'(x) < u_x(x,t;\rho)<0 \mbox{ for } x\in [-1,0),\quad 0<u_x(x,t;\rho)<\Phi'(x) \mbox{ for } x\in (0,1].
\end{equation}
\end{lemma}

\begin{proof}Before presenting the proof, we introduce some notions and notations first. Set
$$
\bar{u}(x,t;r):= \Phi(x)+\bar{c}t + r,
$$
where $\bar{c}$ is the constant in Theorem \ref{thm:exist-GR}/(i).  Clearly, $\bar{u}(x,t;r)$ is an upper solution to Problem \eqref{B} for any $r$.
On the other hand, we define
$$
 \mathcal{E}(t) := \left\{\left. \left( x, \bar{u}(x,t;r) \right) \right| x\in [-1,1],\ r\leq M_1\right\},
$$
Obviously,
$\mathcal{E}(t)$ is the lower
half of the band $\Omega$ with ceiling $\{( x, \bar{u}(x,t;M_1) ) \mid x\in [-1,1] \}$,   and $\mathcal{E}(t)$ moves upward with speed $\bar{c}$.
For each $r\leq M_1$, set
$$
\eta(x,t;r):= u(x,t;\rho)- \bar{u}(x,t;r) ,\quad x\in [-1,1],\ t>0.
$$
Then $\eta$ satisfies a linear parabolic equation. Since the ceiling function $\bar{u}(x,0;M_1) = \Phi(x)+M_1$ of $\mathcal{E}(0)$ lies above $u(x,0;\rho)=\rho(x)$, the comparison principle yields
$$
\eta (x,t;M_1) < 0 , \quad x\in [-1,1],\ t>0.
$$
Hence, $u(x,t;\rho)$ is always immersed in $\mathcal{E}(t)$, and it always
contacts $\bar{u}(x,t;r)$ for some $r\leq M_1$. Moreover,
\begin{itemize}
\item  We call $t_0$ an {\it effective moment} of   $\bar{u}(\cdot,t;r)$  if $\min \eta(\cdot,t_0;r)<0<\max \eta(\cdot,t_0;r)$.
Clearly, in a short period before or after such a moment, $\bar{u}$ always has contact points with $u$.
\item A time interval $(t_1, t_2)$ is called an {\it effective interval} if each $t$ in this interval is an
effective moment.
\item An effective interval $(t_1,t_2)$ is called  {\it maximal} if  for any $\varepsilon_1, \varepsilon_2 >0$,
both $(t_1 - \varepsilon_1, t_2)$ and $(t_1, t_2+\varepsilon_2)$ are not effective intervals.
\end{itemize}

We show (\ref{upper-est}) by studying the derivatives of $\eta(x,t;r)$ at its zeros. The proof is divided into two cases.

\smallskip

\noindent\textbf{Case 1.} We first consider the case when $t=0$.

At this moment, $\bar{u}(x,0;M_1) $ lies above $\rho(x)$ and is tangent to $\rho(x)$ at $x=0$. Thus there is a $M_2<M_1$ such that $u(x,0,\rho)$ contacts $\bar{u}(x,0;M_2)$ at $x=\pm 1$, and for any $r\in (M_2, M_1)$,
$u(x,0,\rho)$ contacts $\bar{u}(x,0;r)$ at exactly two
points $\pm Y(0;r)$. Clearly, there holds
\[
-1<- Y(0;r)<0< Y(0;r)<1,\ \eta_x(Y(0,r))<0,\  \eta_x(-Y(0,r))>0,
\]which implies
\begin{equation}\label{ini-baru-u}
\bar{u}_x (-Y (0;r),0;r) < u_x(- Y(0;r),0;\rho)<0 ,\quad 0<u_x( Y(0;r),0;\rho)< \bar{u}_x (Y(0;r),0;r).
\end{equation}
Hence, (\ref{upper-est}) at $t=0$ follows.

\smallskip

\noindent\textbf{Case 2.} Now we investigate the case when $t>0$.

Since $u(x,0;\rho)\leq\bar{u}(x,0;M_1)$, by the comparison principle we have $u(x,t;\rho)<\bar{u}(x,t;M_1)$
for all $t>0$. Hence, for any given $t^*>0$, $u(x,t^*;\rho)$
immerse in $\mathcal{E}(t^* )$.
So $u(x,t^* ;\rho)$ contacts a family of $\bar{u}(x,t^* ;r)'s$ (for $r<M_1)$.
Therefore, given a point $x^*\in(0,1)$, there is an unique $r^*$ such that $u(x^*,t^* ;\rho)=\bar{u}(x^*,t^* ;r^*)$. Let $(t_1, t_2)$ be a maximal effective interval  of $\bar{u}(\cdot,t;r^*)$.

\smallskip

\textbf{Claim:}
For each
$t\in (t_1, t_2)$, $\bar{u}(\cdot,t;r^*)$ contacts $u(\cdot, t;\rho)$ at exactly two points $\pm Y(t;r^*)$ with $-1<- Y(t;r^*)<0< Y(t;r^*)<1$ and
\begin{equation}\label{t-baru-u}
\bar{u}_x ( -Y (t;r^*),t;r^*) < u_x( -Y(t;r^*),t;\rho)<0 ,\quad 0<u_x( Y(t;r^*),t;\rho)< \bar{u}_x ( Y(t;r^*),t;r^*).
\end{equation}

\smallskip

Note that $t^* \in (t_1,t_2)$  and hence, $\pm Y(t^* ;r^*)=\pm x^*$. So, if the claim is true, then  (\ref{upper-est}) (at $x=x^*$ and $t=t^* $) follows from (\ref{t-baru-u}). Since both $t^* $ and $x^*$ are arbitrary,  we are done. Thus it suffices to show the claim. The proof is divided into three steps.

\medskip
\noindent
{\it Step 1. If $t_1 =0$, then Claim is true}.

\smallskip
First, we consider the case when $r\in (M_2,M_1)$.
According to \eqref{ini-baru-u}, the inequalities hold at $t=0$. By the zero number argument (cf.\,\cite{Ang, Lou2}),
these inequalities remain valid till one of the following happens: (i) the two zero points $\pm Y(t,r^*)$ meet at $x=0$ and then disappear for larger $t$ with $\bar{u} > u$;
(ii) $\bar{u}(x,t;r^*)\leq u(x,t;\rho)$, strictly in $(-1,1)$ and equals hold at $x=\pm 1$. However, in either of these
cases, the time moment will be the end moment of the effective interval $t_2$. Hence,  we are done.

\smallskip

Now we suppose $r=M_2$. Thus, the assumption yields
\begin{equation*}
\bar{u}(x,0;M_2) < u(x,0;\rho) \mbox{ in } (-1,1),\ \bar{u}(\pm 1,0;M_2) = u(\pm 1, 0;\rho),\ \pm [\bar{u}_x (\pm 1,0;M_2) - u_x (\pm 1, 0;\rho)]>0,
\end{equation*}
which implies that
 any $t\in(0, t_2)$ is an effective moment. Hence, the only
possible case is that $\bar{u}(x,t;M_2)$ and $u(x,t;\rho)$ have two non-degenerate contact points $\pm Y(t;r^*)$ near $x=\pm 1$ such that \eqref{t-baru-u} hold.
Then these inequalities remain valid till the end of the effective interval as specified in the above argument.

\medskip
\noindent
{\it  Step 2. If $t_1 >0$, then $\eta(x,t_1;r^*)\geq 0$ in $[-1,1]$, with equality if $x=\pm 1$.}

\smallskip

Note that $t_1$ is not an effective moment. There are only  four  cases:

\begin{itemize}
\item[(i)] If $\eta(x,t_1;r^*)>0$ in $[-1,1]$,  the continuity implies $\eta(x,t;r^*)>0$ in $[-1,1]$ for $0<t_1- t\ll 1$.
This contradicts the definition of $t_1$.

\smallskip

\item[(ii)] If $\max \eta(\cdot,t_1;r^*) =0$, then $u(x,t_1,\rho) \leq \bar{u}(x,t_1;r^*)$ in $[-1,1]$. It follows from the comparison principle that $u(x,t;\rho)<\bar{u}(x,t;r^*)$ for all $x\in [-1,1],\ t>t_1$. This also
contradicts the definition of $(t_1, t_2)$.

\smallskip

\item[(iii)] Suppose $\eta(\pm 1,t_1;r^*) > 0$. Thus, $\max \eta(\cdot , t_1;r^*)>0$, which together with the maximum principle yields $\eta(x,t;r^*)>0$ in
$[-1,1]\times [t_1, t)$, where $t_1 < t \ll t_1 +1$.  This again contradicts the definition of $(t_1, t_2)$.

\item[(iv)] From (i)-(ii), we have $\eta(x,t_1;r^*)\geq 0$ in $[-1,1]$. In particular, (iii) implies $\eta(\pm 1,t_1;r^*)=0$.
\end{itemize}

\medskip
\noindent
{\it Step 3. If $t_1 >0$, then Claim is true.}

\smallskip

We first prove that Claim is true for a short time after $t_1$.
Note that $\bar{u}_x(1,t;r^*)=\Phi'(1)=+\infty$, then
\begin{equation}\label{eta-slope0}
\eta_x(1,t;r^*) = u_x (1,t;\rho) - \bar{u}_x(1,t;r^*) = -\infty ,\quad t>0.
\end{equation}
Since $\eta(1,t_1;r^*)=0$ (see Step 2),  we have $\eta(x,t_1; r^*)>0$ in $[x_1,1)$ for some $x_1$ near $1$. For any $t$ with $t_1 < t\ll t_1 +1$, let $\xi(t)$ be the maximum point of $\eta(\cdot,t;r^*)$. Thus, $\eta(x,t;r^*)>0$ in $\{(x,t)\mid \xi(t)\leq x\leq x_1, 0< t-t_1 \ll 1\}$ due to
Step 2 and the maximum principle. Since $t\in(t_1,t_2)$, according to the convexity and symmetry of $u$ and $\bar{u}$, the two contact points $\pm Y(t;r^*)$ must belong to $[-1,-x_1)\cup(x_1,1]$ and hence, (\ref{t-baru-u}) follows. Therefore, we have proved Claim for a short time after $t_1$.

\smallskip

Now we prove Claim for the whole effective interval $(t_1,t_2)$.
If we can show that $\eta(\pm 1,t;r^*)<0$ for all $t\in (t_1, t_2)$, then Claim follows from the
zero number argument. By the symmetry of $u(x,t;\rho)$, it suffices to show $\eta(1,t;r^*)<0$ in $(t_1, t_2)$.
Suppose by contradiction that there is $t_3\in (t_1, t_2)$ such that
$\eta(1,t;r^*)<0$ for $t\in(t_1, t_3)$  but $\eta(1,t_3;r^*)=0$. The above argument implies that Claim  holds in $(t_1, t_3)$. Due to \eqref{eta-slope0} and
$\eta(1,t_3;r^*)=0$, there exists a point $x_2\in (\xi(t_3), 1)$ such that
$\eta(x,t_3;r^*)>0$ in $x\in (x_2, 1)$. The continuity furnishes a small $\epsilon>0$ such that
$\eta(x_2, t;r^*)>0$ for $t\in  (t_3 -\epsilon, t_3 +\epsilon)$. In the time period
$(t_3 -\epsilon, t_3)$, since $\eta(x_2,t;r^*) \cdot \eta(1,t;r^*)<0$, the unique zero point $Y(t;r^*)$ of $\eta(\cdot,t;r^*)$ must lie in
$(x_2, 1)$. Hence, $\eta(x,t;r^*)>0$ for $(x,t)\in [\xi(t),x_2]\times (t_3-\epsilon, t_3)$.
Now the maximum principle implies $\eta(x,t_3;r^*) >0$ for $x\in [\xi(t_3),x_2]$. Consequently,
$\eta(x,t_3;r^*)>0$ in $[\xi(t_3), 1)$. Thus, $\eta(x,t_3;r^*)$ has no zero points in $(0,1)$, which contradicts the definition of $(t_1, t_2)$. So, $t_3$ must be $t_2$ and therefore, we are done.
\end{proof}

\subsection{Finer lower gradient estimate}

We now give a finer lower gradient estimate of the solution $u(x,t;\rho)$.
Since $a_0 > -b_0$, there exists $h_* >0$ such that $a_0 h > -b_0 \sqrt{1+h^2}$ for all $h> h_*$.
According to Theorem \ref{thm:exist-GR}/(ii), Problem \eqref{TW-p} with $h > h_*$ has a solution pair
$(c(h), \Phi(x;h))$.

Fix any $h^0 > h_* $. By \eqref{3.1}, we can choose $t^0 >0$ large such that
$$
u(x,t^0;\rho)>\Phi(x; h^0)+h^0,\qquad x\in[-1,1].
$$
Thus, $\underline{u}^0 (x,t):=\Phi(x;h^0)+h^0+c(h^0)t$
is a lower solution to  Problem (\ref{B}). The comparison principle furnishes
\begin{equation}\label{4.1}
\underline{u}^0 (x,t)\leq u(x,t+t^0;\rho), \qquad x\in[-1,1],\ t\geq 0.
\end{equation}

On the other hand, choose  $h_0 \in (h_*, h^0)$ and  $H>h^0$ such that
\begin{equation}\label{Phi>u}
\Phi(x;h_0) + h > u(x,t^0;\rho),\quad x\in [-1,1],\ h\geq H.
\end{equation}
It is easy to check that, for any $h\geq H$,
$$
\underline{u}_0 (x,t;h) := \Phi(x;h_0) +c(h_0) t + h
$$
is also a lower solution to Problem (\ref{B}). Denote the union of the graphes of $\underline{u}_0(x,t;h)$'s by
$$
\mathcal{D}(t) :=\left\{\left(x,\underline{u}_0 (x,t;h)\right) \mid x\in[-1,1],\ h\geq  H \right\} =
\left\{(x,y)\mid x\in[-1,1],\ y\geq \underline{u}_0 (x,t;H) \right\}.
$$
Then $\mathcal{D}(t)$ is the upper half of the band $\Omega$ with bottom $\left\{\left(x, \underline{u}_0 (x,t;H)\right)
\mid x\in[-1,1]\right\}$. And $\mathcal{D}(t)$ moves upward with speed $c(h_0)$.

According to the construction above, $\underline{u}^0(x,0)$ lies below $u(x,t^0;\rho)$ while $u(x,t^0;\rho)$ lies below $\mathcal{D}(0)$. Since $c(h^0)>c(h_0)$,
$\underline{u}^0 (x,t)$ rushes into the domain $\mathcal{D}(t)$ for all large $t$. So dose $u(x,t+t^0;\rho)$ by (\ref{4.1}).
Assume this happens for $u(\cdot,t+t^0;\rho)$ when $t\geq T^0$.
Thus, for any $t\geq T^0$, $u(x,t+t^0;\rho)$ contacts a family of $\underline{u}_0 (x,t;h)$'s (for $h\geq H$).
Now set
$\widetilde{\eta}(x,t;h):=u(x,t+t^0;\rho)-\underline{u}_0 (x,t;h)$.
A similar argument as in the proof of Lemma \ref{lem:upper-interior-est} together with $\widetilde{\eta}(x,t;h)$ then furnishes the following result.
\begin{lemma}\label{lem:lower-interior-est}
The gradient of $u$ is bounded from below by $\Phi'(x;h_0)$ in the following sense:
\begin{equation}\label{lower-est}
 u_x(x,t;\rho) < \Phi'(x;h_0) <0 \mbox{ for } x\in [-1,0),\quad 0 < \Phi'(x;h_0) < u_x(x,t;\rho) \mbox{ for } x\in (0,1].
\end{equation}
\end{lemma}

%%%%%%%%%%%%%%%%%%%%%%%%%%%%%%%%%%%%%%%%%

\subsection{Convergence of the solution}

Let $u(x,t; \rho)$ be the symmetric solution with the $u(x,0;\rho)=\rho$, where $\rho$ is defined as in \eqref{varphi-extension}.
Let $\{t_n\}$ be a time sequence with $t_n\rightarrow \infty$ (as $n\rightarrow\infty)$. Set
$$
u_n(x,t):=u(x,t+t_n; \rho)-u(0,t_n;\rho),\qquad x\in[-1,1],\ -t_n<t<\infty.
$$
For any  $\varepsilon>0$ and any $h_0>0$, Lemma \ref{lem:upper-interior-est} together with Lemma \ref{lem:lower-interior-est} yields
\begin{equation}\label{5.1}
\Phi_x (x;h_0)<u_{nx}(x,t)<\Phi'(x),\qquad x\in (0,1-\varepsilon],\ n\gg 1.
\end{equation}
Given any $T>0$,  Lemma \ref{bound-est} together with \eqref{5.1} furnishes
$$
\left\|u_n(x,t) \right\|_{C^{1,0}\left([\varepsilon-1,1-\varepsilon]\times [-T,T]\right)}\leq C_1(\varepsilon,T).
$$
By the $L^p$-estimates, the Sobolev embedding theorem and the Schauder estimate we have
$$
\left\|u_n(x,t) \right\|_{C^{2+\alpha,1+\frac{\alpha}{2}}\left([\varepsilon-1,1-\varepsilon]\times [-T,T]\right)}\leq C_2(\varepsilon,T),\quad \forall \alpha\in(0,1).
$$
Therefore, for any $\beta\in (0,\alpha)$, there exists  a subsequence $\{u_{n_i}\}$ of $\{u_n\}$ and a function
$\mathcal{U}_{T,\varepsilon}\in C^{2+\alpha,1+\frac{\alpha}{2}} $ $\left([\varepsilon-1,1-\varepsilon]\times [-T,T]\right)$
such that
$$
\left\|u_{n_i}-\mathcal{U}_{T,\varepsilon}\right\|_{C^{2+\beta,1+\frac{\beta}{2}}
\left([\varepsilon-1,1-\varepsilon]\times [-T,T]\right)}\rightarrow 0\quad  (i\rightarrow \infty).
$$
Cantor's  diagonal argument then furnishes a function $\mathcal{U}\in C^{2+\alpha,1+\frac{\alpha}{2}}_{\text{loc}}
\left((-1,1)\times \R\right)$ and a subsequence of $\{u_n\}$ (denoted it again by $\{u_{n_i}\}$) such that
 $$
u_{n_i}\rightarrow \mathcal{U}\ \ (i\to \infty), \quad\text{in the }  C^{2+\beta,1+\frac{\beta}{2}}_{\text{loc}}\left((-1,1)\times \R\right)\text{-topology}.
 $$
In particular, $\mathcal{U}(x,t)$ is an entire solution to  (\ref{B})$_1$ (i.e., the equation in (\ref{B})) with $\mathcal{U}(0,0)=0$. Now (\ref{5.1}) implies
$$
\Phi_x(x;h_0)\leq \mathcal{U}_x(x,t)\leq \Phi'(x),\qquad  x\in [0,1),\ t\in \R.
$$
Since
$\lim_{h_0\rightarrow+\infty}\Phi_x(x;h_0)= \Phi' (x)$ (see Theorem \ref{thm:exist-GR}),
we conclude that
$$
\mathcal{U}_x(x,t) =  \Phi'(x),\qquad x\in (-1,1),
$$
which implies
$$
\mathcal{U}(x,t)= \Phi (x)+ C(t),\qquad x\in (-1,1),\ t\in \R.
$$
Since $\mathcal{U}$ satisfies (\ref{B})$_1$, we have
$C'(t)=\bar{c}$ and hence,
$$
\mathcal{U}(x,t)=\Phi (x)+ \bar{c} t,\qquad x\in(-1,1),\ t\in \R.
$$
Here, $\bar{c}$ is the constant in Theorem \ref{thm:exist-GR}/(i).

From above, $\{u_{n_i}\}$ converges to the special cup-like traveling wave $\Phi (x)+\bar{c} t$.
Since this traveling wave is unique and the time sequence $\{t_n\}$ is arbitrarily given, we actually prove the following result.

\begin{theorem}\label{thm:converge-symmetric}
For any $\alpha\in (0,1)$,
\begin{equation}\label{5.2}
u(x,t+s;\rho)-u(0,s;\rho)\rightarrow \Phi (x)+ \bar{c} t, \qquad \text{as} \quad s\rightarrow \infty,
\end{equation}
in the $C^{2+\alpha,1+\frac{\alpha}{2}}_{\text{loc}}\left((-1,1)\times \R \right)$-topology.
\end{theorem}

%\begin{remark}\label{rem:symm}
%The symmetric assumption on $a$ and $b$ is {\it only} used in subsection 4.3 and 4.4, which helps to
%establish the interior gradient estimates (i.e., Lemmas \ref{lem:upper-interior-est}-\ref{lem:lower-interior-est}). These estimates
%imply the convergence of $u$ to $\Phi(x)+\bar{c}t$ as in Theorem \ref{thm:converge-symmetric}. To give the interior gradient estimate for general solutions, we also need the symmetric assumptions in proofing the following Lemma \ref{lem4.3}
%\end{remark}

\section{General solutions}\label{sec5}
In this section we consider Problem \eqref{B} with general initial data and complete the proof of Theorem \ref{thm:main}. In the following, we assume $a$ and $b$ are even functions as in \eqref{symm cond}.

\subsection{Interior estimates}

Let $\rho(x)$ be defined as in \eqref{varphi-extension}. Hereinafter, we additionally require that $(p,M_1)$ satisfies
\begin{equation}\label{new5.1}
\rho(x) < u_0(x) < u(x,T;\rho),\qquad x\in[-1,1],
\end{equation}
for some positive $T$. Then  the comparison principle yields
\begin{equation}\label{4.9}
u(x,t;\rho ) < u(x,t; u_0)< u(x,t+T;\rho ) ,\qquad x\in [-1,1],\ t>0.
\end{equation}

In the sequel,  we will present a uniform interior gradient estimate. In order to do this, we need the following lemma.

\begin{lemma}\label{lem4.3}
For any small $\varepsilon\in(0, \frac12)$ and any $t>0$, there hold
\begin{equation}\label{4.10}
\min_{1-2\varepsilon\leq x\leq 1-\varepsilon}\left|u_x(x,t;u_0)\right|< M_2 :=\varepsilon^{-1} \Big[\Phi(1-\varepsilon)+{\bar{c}} T \Big],\qquad
\min_{\varepsilon-1 \leq x\leq 2\varepsilon-1}\left|u_x(x,t;u_0)\right|<M_2,
\end{equation}
where  $\bar{c}$ (resp., $T$) is defined as in Theorem \ref{thm:exist-GR} (resp., \eqref{new5.1}).
\end{lemma}

\begin{proof}
We only show the first inequality since the second one can be  proved similarly.
Assume by contradiction that, for some $t=t_0>0$,
$$
\left|u_x(x,t_0;u_0)\right| \geq M_2, \qquad x\in[1-2\varepsilon, 1-\varepsilon].
$$
Integrating this inequality over $[1-2\varepsilon, 1-\varepsilon]$, we obtain
\begin{equation}\label{4.11}
\Phi\left(1-\varepsilon\right)+\bar{c} T\leq u\left(1-\varepsilon, t_0; u_0\right)-u\left(1-2\varepsilon, t_0; u_0\right).
\end{equation}

On the other hand, (\ref{4.9}) together with the convexity of $u(\cdot, t; \rho )$ yields
\begin{equation}\label{4.12}
 \begin{array}{lll}
 u \left(1-\varepsilon, t_0; u_0\right)-u\left(1-2\varepsilon, t_0; u_0\right)
 & < & u\left(1-\varepsilon, t_0+T; \rho \right)-u\left(1-2\varepsilon, t_0; \rho \right) \\
 & \leq & u\left(1-\varepsilon, t_0+T; \rho \right)-u\left(0, t_0; \rho \right).
 \end{array}
\end{equation}
Since $u(x,t_0;\rho )<\Phi(x)+u(0,t_0;\rho )$, the comparison principle  furnishes
$$
u(x,t_0 +T;\rho )<\Phi(x)+\bar{c} T+u(0,t_0;\rho ),
$$
which implies
$$
u\left(1-\varepsilon, t_0+T; \rho \right)-u\left(0, t_0; \rho \right)< \Phi(1-\varepsilon)+\bar{c} T.
$$
However, this contradicts (\ref{4.11}) and (\ref{4.12}). Hence, the lemma follows.
\end{proof}

With the help of the above lemma, we obtain the following interior gradient estimate.

\begin{lemma}\label{lem:interior-g-est}
For any small $\varepsilon>0$, there exists $T_\varepsilon >0$ such that
\begin{equation}\label{interior-g-est}
 |u_x (x,t; u_0)| \leq M_3, \quad   -1+2\varepsilon <x <1-2\varepsilon,\ t>T_\varepsilon,
 \end{equation}
where $M_3 := \max\{M_2, \|u_x(\cdot, T_\varepsilon; u_0)\|_{L^\infty}\}$ and $M_2$ is defined as in \eqref{4.10}.
\end{lemma}

\begin{proof}
Since $u(x,t;u_0) \to \infty$ as $t\to \infty$, there exists a large $ T' >0$  such that $u(\pm 1, t; u_0)
> M_2$ for all $t>T' $. Set $\zeta(x,t):= u_x(x,t; u_0)-M_2$. Then $\zeta$ satisfies
$$
\zeta_t = a_1 \zeta_{xx} +b_1  \zeta_x,\quad -1<x<1,\ t>0,
$$
for some bounded $a_1, b_1$ and $\zeta (1,t)>0>\zeta(-1,t)$ for $t>T'$. Using the zero number properties (cf. \cite{Ang})
we conclude that, for some $T_\varepsilon  >T' $, the function $\zeta(\cdot, t)$ has only non-degenerate zeros for
$t\geq T_\varepsilon $. Denote the largest zero of $\zeta(\cdot,t)$ in $(-1,1)$ by $ \rho_+(t)$. The non-degeneracy of $\rho_+(t)$ implies that $x= \rho_+(t)$ is a continuous curve. Moreover,
\eqref{4.10} indicates that $\rho_+(t) > 1- 2\varepsilon$.
In a similar way one can find another continuous curve $x=\rho_-(t)$ for $t>T_\varepsilon $ ($T_\varepsilon $ can be chosen larger if necessary)
such that $\rho_-(t) \in (-1, -1 + 2\varepsilon)$ and $u_x(\rho_-(t),t) = - M_2$ for $t >T_\varepsilon$.
Thus, using the maximum principle for $u_x$ in the domain $D(T_\varepsilon ) := \{(x,t) \mid \rho_-(t)<x<\rho_+(t),\ t>T_\varepsilon \}$,
we conclude that $|u_x (x,t;u_0)|\leq M_3$ in $D(T_\varepsilon)$. Now the estimate \eqref{interior-g-est}  follows
from the fact that $\rho_-(t)<-1+2\varepsilon < 1-2\varepsilon < \rho_+(t)$ for $t>T_\varepsilon$.
\end{proof}

\subsection{Convergence of general solutions}

Given any time sequence $\{t_n\}$ with $t_n\rightarrow \infty$, we consider the solution sequence $\{ u(x,t+t_n; u_0)-u(0,t_n;\rho )\}$.

For any given small $\varepsilon>0$ and any $\tau >0$, let $T_\varepsilon$ be as in Lemma \ref{lem:interior-g-est}.
Thus \eqref{4.9} together with \eqref{interior-g-est} implies that, for all large $n$,  the $C^{1,0}([2\varepsilon-1, 1-2\varepsilon]\times [-\tau, \tau])$-norms of $u(x,t+t_n; u_0)-u(0,t_n;\rho )$'s are bounded,
which are independent of $n$. For any $\alpha\in (0,1)$,
the standard parabolic theory furnishes the $C^{2+\alpha,1+\frac{\alpha}{2}}
([2\varepsilon-1, 1-2\varepsilon] \times [-\tau, \tau])$-bounds for the solution sequence, which are also independent of  $n$. Hence,
we can choose   a convergent subsequence.
By taking $\varepsilon \to 0$ and $\tau \to \infty$ and using Cantor's diagonal argument, we obtain
a subsequence $\{t_{n_k}\}$ of $\{t_{n}\}$ such that (as $k\to \infty$),
\begin{equation}\label{u-n-to-W}
u(x,t+t_{n_k}; u_0)-u(0,t_{n_k};\rho )\rightarrow \mathcal{W}(x,t) \qquad \text{in the }
C^{2+\alpha,1+\frac{\alpha}{2}}_{\text{loc}}\left((-1,1)\times \R \right)\text{-topology},
\end{equation}
for some entire solution $\mathcal{W}$ to (\ref{B})$_1$ (i.e., the equation in (\ref{B})).
On the other hand, as a consequence of Theorem \ref{thm:converge-symmetric},  we have (as $k\to \infty$),
$$
u(x,t+t_{n_k};\rho )-u(0, t_{n_k};\rho )\rightarrow \Phi(x)+\bar{c} t, \qquad u(x,t+T+t_{n_k};\rho )
-u(0,t_{n_k};\rho )\rightarrow \Phi(x)+\bar{c} (t+T),
$$
where $T$ is defined as in \eqref{new5.1}.
Hence, from (\ref{4.9}) we derive
\begin{equation}\label{5.6}
\Phi(x)+\bar{c} t\leq\mathcal{W}(x,t)\leq \Phi(x)+\bar{c} t+\bar{c} T,\qquad x\in(-1,1),\ t\in \R.
\end{equation}

Let
$$
\theta(x,t):= \arctan u_x(x,t; u_0),\quad x\in [-1,1],\ t>0.
$$
Then $\theta$ satisfies
\begin{equation}\label{eq-theta}
\left\{
 \begin{array}{ll}
 \theta_t = a(\tan \theta) \cos^2 \theta \cdot \theta_{xx} + a'(\tan \theta) \theta_x^2 + b(\tan\theta)\sin \theta \cdot \theta_x +
 b'(\tan\theta) \frac{\theta_x}{\cos\theta},  & -1<x<1,\ t>0,\\
 \\
 \theta(\pm 1,t)=  \pm \arctan u(\pm 1,t), & t>0.
 \end{array}
 \right.
\end{equation}
The global existence of $u$ (see Theorem \ref{thm:global}) implies that $\theta$ is well-defined for all $t>0$.
Moreover, by \eqref{u-n-to-W} we have
$$
\theta(x,t_n +t) \to  \arctan \mathcal{W}_x (x,t) \mbox{ in the } C^{2+\alpha, 1+\alpha/2}_{\text{loc}} ((-1,1)\times \R)
\mbox{-topology}.
$$
It follows from the zero number argument (see  Du and Matano \cite[Section 3]{DM})
that any $\omega$-limit of the bounded solution $\theta$ is a stationary one.
Therefore, for each fixed $\tau\in \R$, $ \arctan \mathcal{W}_x (x,\tau)$ (which is an $\omega$-limit of $\theta$) satisfies
$ [\arctan \mathcal{W}_x ]_t (x,\tau) \equiv 0$, that is, $\mathcal{W}_{xt} (x,\tau) \equiv 0$. Since $\tau$ is arbitrary, then $\mathcal{W}_{xt} (x,t) \equiv 0$ in $(-1,1)\times \R$. Therefore,
$$
\mathcal{W} (x,t) = P(x) + Q(t), \quad x\in (-1,1),\ t\in \R,
$$
for some functions $P(x)$ and $Q(t)$. Substituting it into the equation of $\mathcal{W}$ we have
$$
Q'(t) = a(P_x ) \frac{P_{xx}}{1+P_x^2} + b(P_x) \sqrt{1+P_x^2},\quad x\in (-1,1), \ t\in \R.
$$
Note that the left-hand side of the equation above is only dependent on $t$ while the right-hand side is only  dependent on $x$. Hence, both sides are the same constant, say $c_1$. Then $Q(t) = c_1 t +C_1$ and therefore, $\mathcal{W}(x,t)$ is a traveling wave $P(x) + c_1 t +C_1$.
Then \eqref{5.6} implies $c_1 =\bar{c}$. Hence, $P(x)$ is nothing but $\Phi(x)+C_2$. So
\begin{equation}\label{Wsiwhat}
\mathcal{W}(x,t) = \Phi(x) +\bar{c}t +C_3,
\end{equation}
for some $C_3 \in [0,\bar{c}T]$ (by \eqref{5.6}).

\medskip
\noindent
\begin{proof}[Proof of Theorem \ref{thm:main}] The existence of the solution to Problem \eqref{B} follows from Theorem \ref{thm:global}.
It remains to prove (\ref{1.5thereom2secnd}).
Suppose by contradiction that (\ref{1.5thereom2secnd}) is not true.
Then one could find an $\varepsilon_0>0$  and  a sequence $\{t_n\}$ such that $t_n>n$ and
\begin{equation}\label{contradbasis}
\left|\left[u(x,t+t_n;u_0) - u(0,t_n;u_0)\right] -\left[ \Phi(x) +\bar{c} t\right]\right|\geq\varepsilon_0.
\end{equation}
In particular, (\ref{contradbasis}) holds for any subsequence $\{t_{n_k}\}$.

On the other hand, in view of (\ref{u-n-to-W}) and (\ref{Wsiwhat}),  there exists a subsequence $\{t_{n_k}\}$  such that
$$
u(x,t+t_{n_k}; u_0) - u(0,t_{n_k};\rho) \to \Phi(x)+\bar{c} t +C,\quad \mbox{ as } k\to \infty.
$$
Thus, we have
\begin{align*}
&u(x,t+t_{n_k};u_0) - u(0,t_{n_k};u_0)\\
 =& [u(x,t+t_{n_k};u_0) - u(0,t_{n_k};\rho)] - [u(0,t_{n_k};u_0) - u(0,t_{n_k};\rho)]
\to \Phi(x) +\bar{c} t, \quad \mbox{ as } k\to \infty,
\end{align*}
 which contradicts (\ref{contradbasis}). Hence, (\ref{1.5thereom2secnd}) is true.
\end{proof}
%%%%%%%%%%%%%%%%%%%%%%%%%%%%%%%%%%%%%%%%%%%%%%%%%%%
%%%%%%%%%%%%%%%%%%%%%%%%%%%%%%%%%%%%%%%%%%%%%%%%%%%%
%%%%%%%%%%%%%%%%%%%%%%%%%%%%%%%%%%%%%%%%%%%%%%%%%%%%
%%%%%%%%%%%%%%%%%%%%%%%%%%%%%%%%%%%%%%%%%%%%%%%%%%%

\noindent{\textbf{Acknowledgements}}
 This work was supported by NNSFC (No. 11761058) and NSFS (No. 17ZR1420900, No. 19ZR1411700).  The authors are greatly indebted to Pro. B. Lou for many
useful discussions and  helpful comments.


\begin{thebibliography}{99}





\bibitem{MD}
 M .Alfaro, D .Hilhorst, H .Matano, \emph{The singular limit of the Allen-Cahn equation and the FitzHugh-Nagumo system}. J. Differential Equations, \textbf{245} (2008), no. 2, 505-565.



\bibitem{AW1}
S.J.~Altschuler and L.F.~Wu,
\emph{Convergence to translating solutions for a class of quasilinear parabolic boundary problems},
   Math. Ann., \textbf{295} (1993), 761-765.

\bibitem{AW2}
   S.J.~Altschuler and L.F.~Wu,
   \emph{Translating surfaces of the non-parametric mean curvature flow with prescribed contact angle},
    Calc. Var. Partial Differential Equations,  \textbf{2} (1994), 101-111.


\bibitem{Ang}
S.B.~Angenent, \emph{The zero set of a solution to  a parabolic equation}, J. Reine Angew. Math., \textbf{390} (1988), 79-96.




\bibitem{CaiLou2}
 J.~Cai and B.~Lou,
 \emph{Convergence in a quasilinear parabolic equation with time almost periodic boundary conditions},
Nonl. Anal., \textbf{75} (2012), 6312-6324.


\bibitem{CGK}
Y.L.~Chang, J.S.~Guo and Y.~Kohsaka,
\emph{On a two-point free boundary problem for a quasilinear parabolic equation},
Asymptotic Anal., \textbf{34} (2003), 333-358.


\bibitem{CX}
X. Chen, \emph{Generation and propagation of interfaces for reaction-diffusion equations}, J. Differential Equations, \textbf{96} (1992), 116-141.


\bibitem{ChenGuo}
X. Chen and J.-S. Guo,
\emph{Motion by curvature of planar curves with end points moving freely on a line},
Math. Ann., \textbf{350} (2011), 277-311.

\bibitem{CW}
K.-S.~Chou and X.L.~Wang,
\emph{ The curve shortening problem under Robin boundary condition},
NoDEA Nonlinear Differential Equations Appl., \textbf{19} (2012), 177-194.

%\bibitem{ChouZhu}
%K.-S. Chou and X.-P. Zhu,
%The curve shortening problem, Chapman \& Hall/CRC, Boca Raton, FL, 2001.
%
\bibitem{DM}Y.~Du and H.~Matano,
\emph{Convergence and sharp thresholds for propagation in nonlinear diffusion problems},
J. Eur. Math. Soc., \textbf{12} (2010), 279-312.
%





%\bibitem{ES}
%L.C. Evans and J. Spruck,
%\emph{Motion of level sets by mean curvature III},
%J. Geom. Anal., \textbf{2} (1992), 121-150.

\bibitem{GGH}
M.-H. Giga, Y. Giga and H. Hontani,
\emph{Self-similar expanding solutions in a sector for a crystalline flow},
SIAM J. Math. Anal., \textbf{37} (2005), 1207-1226.

\bibitem{GH}
J.-S.~Guo and B.~Hu,
\emph{On a two-point free boundary problem},
 Quart. Appl. Math., \textbf{64} (2006), 413-431.

\bibitem{GMSW}
J.-S.~Guo, H. Matano, M. Shimojo and C.H. Wu,
\emph{On a free boundary problem for the curvature flow with driving force},
Arch. Ration. Mech. Anal., \textbf{219} (2016), 1207-1272.


\bibitem{GM}
M. Gurtin,  Thermomechanics of Evolving Phase Boundaries in the Plane, Oxford, Clarendon Press, 1993.


\bibitem{Hui}
G. Huisken,
\emph{Nonparametric mean curvature evolution with boundary conditions},
J. Differential Equations,  \textbf{77} (1989), 369-378.

\bibitem{Koh}
 Y.~Kohsaka,
\emph{Free boundary problem for quasilinear
     parabolic equation with fixed angle of contact to a boundary},
 Nonl. Anal., \textbf{45} (2001), 865-894.

\bibitem{Lou1} B.~Lou,
\emph{Periodic traveling waves of a mean curvature flow in heterigeneous media}, Discrete Contin. Dynam. Syst., \textbf{15} (2009),
231-249.


\bibitem{Lou2}
B.~Lou,
\emph{The zero number diminishing property under general boundary conditions}, Appl. Math. Lett., \textbf{95} (2019), 41-47.

\bibitem{LMN}B.~Lou, H.~Matano and K.~Nakamura,
\emph{Recurrent traveling waves in a two-dimensional saw-toothed cylinder
and their average speed}, J. Differential Equations, \textbf{255} (2013), 3357-3411.

\bibitem{LWY}
B. Lou, X. Wang and L. Yuan,
\emph{Convergence to a grim reaper for a curvature flow with variable boundary slopes}, preprint.

\bibitem{MNL}
 H.~Matano, K.I.~Nakamura and B.~Lou,
   \emph{Periodic traveling waves in a two-dimensional cylinder with saw-toothed boundary and their homogenization limit},
 Netw. Heterog. Media, \textbf{1} (2006), 537-568.

\bibitem{MW}
W. W. Mullins, \emph{Two-dimensional motion of idealized grain boundaries}, J. Appl. Phys., \textbf{27} (1956), 900-904.



\bibitem{ORS}
 N.C.~Owen, J.~Rubinstein and P.~Sternberg,
 \emph{Minimizers and gradient flows for singularly perturbed bi-stable potentials with a Dirichlet condition},
 Proc. Roy. Soc. London Ser. A, \textbf{429} (1990), 505-532.


\bibitem{RJ}
 J. Rubinstein, P. Sternberg and J. B. Keller, \emph{Fast reaction, slow diffusion, and curve shortening}, SIAM J. Appl. Math., \textbf{49} (1989), 116-133.



%\bibitem{RSK}
% J. Rubinstein, P. Sternberg and J. Keller,
% \emph{Fast reaction, slow diffusion, and curve shortening}, SIAM J. Appl. Math. \textbf{49} (1989), 116-133.
%
\bibitem{YuanLou} L.~Yuan and B.~Lou,
 \emph{Entire solutions of a mean curvature flow connecting two periodic traveling waves},
Appl. Math. Lett., \textbf{87} (2019), 73-79.

\end{thebibliography}
\end{document}